\documentclass[11pt,a4paper]{article}
\usepackage[utf8]{inputenc}
\usepackage{amsmath}
\usepackage{amscd}
\usepackage{amsthm}
\usepackage{amssymb}
\usepackage{tikz-cd}
\usepackage{amsfonts}
\usepackage{hyperref}
\usepackage{url}
\usepackage{times}
\newtheorem{theorem}{Theorem}
\newtheorem{lemma}{Lemma}[section]
\newtheorem{note}{Note}[section]
\newtheorem{remark}{Remark}[section]
\allowdisplaybreaks

\title{Generalization on the higher moments of the Fourier coefficients of symmetric power $L$-functions} 
\author{K. Venkatasubbareddy\\Email: \href{venkatasubbareddy7313@gmail.com}{venkatasubbareddy7313@gmail.com}
\\Department of Mathematical Sciences, IISER Berhampur,\\
Berhampur, Odisha, India-760003}
\date{}

\begin{document}
\maketitle
Dedicated to Prof. A. Sankaranarayanan on the occasion of his 65th birthday.

\begin{abstract}
For an even integer $k\geq 2$, let $f$ be a primitive holomorphic cusp form of weight $k$ for the full modular group $SL(2,\mathbb{Z})$ and let $\lambda_{{\rm{sym}}^jf}(n)$ denote the $n^\text{th}$ normalized Fourier coefficient of the $j^{\text{th}}$ symmetric power $L$-function $L(s,{\rm{sym}}^j f)$. It has been an interesting problem to study the average behaviour of $\lambda_{{\rm{sym}}^jf}(n)$ and their higher powers, and many researchers in the literature have studied the sum
\begin{equation*}
    \sum_{n\leq x} \lambda_{{\rm{sym}}^j}^l(n),
\end{equation*}
for various values of $l$ and $j$. In this paper, we improve and generalize previously known results concerning the sum above for positive integers $l$ and $j$ such that $lj\geq 4$.
\end{abstract}

\footnote{2020 AMS \emph{Mathematics subject classification.} Primary 11F11, 11F30, 11M06.}

\section{Introduction}
For an even integer $k\geq 2$, let $f$ be a primitive holomorphic cusp form of weight $k$ for the full modular group $SL(2,\mathbb{Z})$. Here, the term primitive refers that $f$ is an eigenfunction of all Hecke operators simultaneously. Throughout the paper, we refer to $f$ as a primitive holomorphic cusp form and $H_k$ as the set of all primitive holomorphic cusp forms of weight $k$ for the full modular group $SL(2,\mathbb{Z})$. It is well known that $f(z)$ has a Fourier series expansion at the cusp $\infty$ as
\begin{equation*}
    f(z)=\sum_{n=1}^\infty \lambda_f(n)n^{(k-1)/2}e^{2\pi i n z}
\end{equation*}
for $\Im (z)>0$, where $\lambda_f(n)$ are the normalized Fourier coefficients satisfying the multiplicative property that 
\begin{equation*}
    \lambda_f(m)\lambda_f(n)=\sum_{d|(m,n)}\lambda_f(\frac{mn}{d^2}) 
\end{equation*}
for all integers $m,n\geq 1$. In 1974, Deligne \cite{Deligne} proved the Ramanujan-Petersson conjecture that $|\lambda_f(n)|\leq d(n)$, where $d(n)$ is the divisor function and which is equivalent to say that for each prime $p$, there exist two complex numbers, namely $\alpha_p$ and $\beta_p$ such that
\begin{equation*}
    \alpha_p\beta_p=|\alpha_p|=|\beta_p|=1 \text{ and } \lambda_f(p)=\alpha_p+\beta_p.
\end{equation*}

For integers $j\geq 1$, the $j^{\text{th}}$ symmetric power $L$-function attached to $f$ is defined as
\begin{equation*}
        L(s, {\rm{sym}}^jf)=\prod_p\prod_{m=0}^j(1-\alpha_p^{j-2m}p^{-s})^{-1}=\sum_{n=1}^\infty \frac{\lambda_{{\rm{sym}}^j f}(n)}{n^s}
\end{equation*}
for $\Re (s)>1$. It is well known that $\lambda_{{\rm{sym}}^j f}(n)$ is a real multiplicative function and $\lambda_{{\rm{sym}}^j f}(p)=\lambda_f(p^j)$ for each prime $p$ and integers $j\geq 1$. Note that $\displaystyle L(s,{\rm{sym}}^0 f)=\zeta(s)$ (Riemann zeta function) and $L(s,{\rm{sym}}^1 f)=L(s,f)$ (Hecke $L$-function).

From the works of Gelbart and Jacquet \cite{Gelbart and Jacquet1978}, Kim and Shahidi \cite{Kim 2003, Kim and Shahidi2002, Kim and Shahidi2002 II}, the $L$-functions $L(s,{\rm{sym}}^jf)$ for $j=1,2,3,4 $ admit analytic continuation and satisfy a functional equation of the Riemann zeta type. In fact, from the recent work of Newton and Thorne \cite{Newton and Thorne2021, Newton and Thorne2021 II} on the symmetric power lifts of Hecke eigenforms, all the $L$-functions $L(s,{\rm{sym}}^jf)$, for $j\geq 1$ are holomorphic and satisfy a functional equation of Riemann zeta type. Thus, for $j\geq 1$, $L(s,{\rm{sym}}^jf)$ are general $L$-functions in the sense of Perelli.

It has been an interesting problem to study the sum 
\begin{equation*}
    \sum_{n\leq x} \lambda_{{\rm{sym}}^j}^l(n)
\end{equation*}
for positive integers $l$ and $j$. In this direction, many authors have contributed to this problem in the literature. For example, in 2008, Fomenko \cite{Fomenko2008} studied the sum for $l=j=2$, and proved that 
\begin{equation*}
    \sum_{n\leq x} \lambda_{{\rm{sym}}^2}^2(n)=\mathcal{C}_1x+O(x^\gamma)
\end{equation*}
for some $\gamma<1$. Later, this result has been improved and generalized by various authors; see \cite{He}, \cite{Lao2010}, \cite{Lao2012}. In 2019, Sankaranarayanan et al. \cite{AS Singh and Srinivas} proved that 
\begin{align*}
   & \sum_{n\leq x} \lambda_{{\rm{sym}}^3}^2(n)=\mathcal{C}_3x+O(x^{\frac{15}{17}+\varepsilon})\\
   & \sum_{n\leq x} \lambda_{{\rm{sym}}^4}^2(n)=\mathcal{C}_4x+O(x^{\frac{12}{13}+\varepsilon}).
\end{align*}
In 2021, Luo et al. \cite{Luo Lao and Zou} established that
\begin{equation*}
   S_l(x):= \sum_{n\leq x} \lambda_{{\rm{sym}}^2}^l(n)=xP_l(\log x)+O(x^{\theta_l+\varepsilon}),
\end{equation*}
where $P_l(t)$ is a polynomial in $t$ with deg $P_3 = 0$, deg $P_4 = 2$, deg $P_5 = 5$, deg $P_6 = 14$, deg $P_7 = 35$ and deg $P_8 = 90$, and the exponents $\theta_l$ are given 
by 
\begin{align*}
    \theta_3=\frac{971}{1055},\theta_4=\frac{262}{269},\theta_5=\frac{3237}{3265},\theta_6=\frac{4923}{4937},\theta_7=\frac{7442}{7449},\theta_8=\frac{89771}{89779}.
\end{align*}
Furthermore, in the same paper, they established the following asymptotic formulae
\begin{equation*}
    T_j(x):=\sum_{n\leq x} \lambda_{{\rm{sym}}^j}^2(n)=\Tilde{c}_jx+\begin{cases}
        O(x^{\theta_j^*+\varepsilon}),& j=3,4,5,6\\
        O(x^{\theta_j^*}),& j=7,8,
    \end{cases}
\end{equation*}
where $\Tilde{c}_j$ are some constants, and the exponents $\theta_j^*$ are given by
\begin{align*}
    \theta_3^*=\frac{551}{635},\theta_4^*=\frac{929}{1013},\theta_5^*=\frac{1391}{1475},
    \theta_6^*=\frac{979}{1021},\theta_7^*=\frac{63}{65},\theta_8^*=\frac{40}{41}.
\end{align*}
In 2023, Liu \cite{Liu2023} has improved the error term bounds for both $S_l(x)$ and $T_j(x)$, and the improved exponents are as follows.
\begin{align*}
    \theta_3=\frac{1367}{1487},\theta_4=\frac{1483}{1523},\theta_5=\frac{459}{463},
    \theta_6=\frac{12237}{12272},\theta_7=\frac{74069}{74139},\theta_8=\frac{335197}{335302}
\end{align*}
and 
\begin{align*}
    \theta_2^*=\frac{389}{509},\theta_3^*=\frac{779}{899},\theta_4^*=\frac{1319}{1439},\theta_5^*=\frac{1979}{2099},
    \theta_6^*=\frac{2759}{2879},\theta_7^*=\frac{3659}{3779},\theta_8^*=\frac{4679}{4799}.
\end{align*}
In this paper, we further improve the above exponents of Liu; moreover, we generalize the work to the sum 
\begin{equation*}
    \sum_{n\leq x} \lambda_{{\rm{sym}}^j}^l(n)
\end{equation*}
for positive integers $l$ and $j$ such that $lj\geq 4$. Throughout the paper, we assume that $l$ and $j$ are positive integers and $\varepsilon$ is some small positive, which need not be the same at each occurrence.

Precisely, we prove:
\begin{theorem}
    For any $\varepsilon>0$, we have
    \begin{equation*}
        \sum_{n\leq x} \lambda_{{\rm{sym}}^j}^l(n)=\begin{cases}
            xP_{d_{\frac{lj}{2}}-1}(\log x)+O\left(x^{\theta_{l,j}+\varepsilon}\right)&\text{ if $lj$ is even}\\
            O\left(x^{\theta_{l,j}+\varepsilon}\right)&\text{ if $lj$ is odd},
        \end{cases}
    \end{equation*}
    where $P_k(y)$ is a polynomial of degree $k$ in $y$, and 
    \begin{equation*}
        \theta_{l,j}=
        \begin{cases}
        1-\frac{126j^{3/2}}{63(D-d_2)j^{3/2}+16\sqrt{15}d_2}&\text{ if $lj=4$}\\
        1-\frac{630j^{3/2}}{j^{3/2}(315D-315d_{\frac{lj}{2}}-189d_{\frac{lj}{2}-1})+80\sqrt{15}d_{\frac{lj}{2}}}&\text{ if $lj\geq6$ is even}\\
        1-\frac{6}{3D-2e_{\frac{lj-1}{2}}}&\text{ if $lj\geq5$ is odd},
        \end{cases}
    \end{equation*}
    $D=(j+1)^l$ is the degree of the $L$-function $L_{l,j}(s)$ defined in Lemma \ref{L2.3}, $d_i$'s and $e_i$'s are as in Lemma \ref{L2.2}.\label{T1}
\end{theorem}

\begin{remark}
   We note that our method does not apply in the case $lj \leq 3$. In these cases, the decomposition of the $L$-function $L_{l,j}(s)$ contains at most two factors, which creates a barrier to applying subconvexity bounds. Consequently, we must instead rely on the Cauchy--Schwarz or H\"older inequality, which results in a loss.
\end{remark}

\begin{remark}
   It is not difficult to obtain a slight further improvement of our results in the case when $lj$ is even, by moving the line of integration to $\Re(s)=1-\sigma(j)$ with $\sigma(j)<\frac{1}{j^3}$, and then following the same arguments as in the proof of Theorem \ref{T1}. For example, in the case $l=j=2$ ($d_2=1, D=9$), if we move the line of integration to $1-\frac{1}{j^{18}}$, the error term can be refined to $x^{\frac{3}{4}+\delta+\varepsilon}$ with $\delta<10^{-7}$. Similarly, in the case $lj\geq 6$, by moving the line of integration to $\Re(s)=1-\frac{1}{j^a}$, for $a\geq 4$ and following the same arguments as in the proof of Theorem \ref{T1}, we can refine the error terms to $x^{\theta_{l,j}^*+\varepsilon}$ with
    \begin{equation*}
        \theta_{l,j}^*=1-\frac{630}{(315D-315d_{\frac{lj}{2}}-189d_{\frac{lj}{2}-1})}+\delta,
    \end{equation*}
    where $\delta:=\delta(l, j, a)$ is too close to zero for large $a$.  
\end{remark}
The following table indicates the previous exponents, the values of $\theta_{l,j}$ from Theorem \ref{T1} and the possible refined error terms $\theta_{l,j}^*$ for some values of $l$ and $j$.
\begin{table}[!ht]
    \begin{tabular}{r|c|c|c}
       $j=2$, $l=$&\text{Previous exponents}& $\theta_{l,j}$&$\theta_{l,j}^*$   \\\hline
       2 &0.7642...& 0.7604...&0.75\\
       3  &0.9193...& 0.9185...&0.91735537...\\
        4  & 0.9737...& 0.9734...&0.97311827...\\
        5  & 0.99136...&0.991307... &0.99122807...\\
       6   &0.99714...& 0.997133...&0.99711149...\\
        7  &0.9990558...&0.9990516... &0.99904598...\\
        8  &0.9996868...&0.9996852... &0.99968408...
    \end{tabular}
    \end{table}
  \begin{table}[!ht]
    \begin{tabular}{r|c|c|c}
      $l=2$, $j=$ &\text{Previous exponents}& $\theta_{l,j}$&$\theta_{l,j}^*$  \\\hline
      2 &0.764...& 0.7604...&0.75\\
      3  &0.866...& 0.8629...&0.86111111...\\
        4  & 0.916...& 0.9149...&0.91452991...\\
        5  &0.9428...&0.9420... &0.94186046...\\
       6   &0.9583...&0.957865...&0.95780590...\\
        7  &0.9682...&0.967975... &0.96794871...\\
        8  & 0.97499...&0.9748248... &0.97481108...
    \end{tabular}
\end{table}
Our results improve upon those of \cite{Liu2023} and \cite{Luo Lao and Zou}.
\section{Lemma's}
\begin{lemma}\label{L2.1}
    For any $x$ we have
    \begin{equation}
        (1+x^2+x^4+\cdots+x^{2j})^l=\sum_{m=0}^{lj}c_mx^{2m},\label{E1}
    \end{equation}
    where $c_{m}:=c_m(l,j)$ denotes the number of integral solutions of the equation 
    \begin{equation*}
        m_1+m_2+\cdots+m_l=m\qquad 0\leq m_i\leq j. 
    \end{equation*}
    By the inclusion-exclusion principle, 
    \begin{equation}
       c_{m}=
\sum_{r=0}^{\left\lfloor \frac{m}{\,j+1\,} \right\rfloor}
(-1)^{r}\binom{l}{r}
\binom{m-r(j+1)+l-1}{\,l-1\,},
\qquad 0\le m\le lj \label{E2}
    \end{equation} 
    with $\binom{n}{r}=0$ if $n<r$.
    Moreover, the coefficients $c_m$'s are Palindromic, that is, $c_m=c_{lj-m}$ and satisfy the chain of inequalities $c_0\leq c_1\leq c_2\leq \cdots\leq c_{\frac{lj}{2}}$ and $c_{\frac{lj}{2}}\geq c_{\frac{lj}{2}+1}\geq \cdots\geq c_{lj-1}\geq c_{lj}$ when $lj$ is even and $c_0\leq c_1\leq c_2\leq \cdots\leq c_{\frac{lj-1}{2}}$ and $c_{\frac{lj-1}{2}}\geq c_{\frac{lj+1}{2}}\geq \cdots\geq c_{lj-1}\geq c_{lj}$ when $lj$ is odd.
\end{lemma}
\begin{proof}
   We write 
   \begin{align*}
        (1+x^2+x^4+\cdots+x^{2j})^l&=\left(\sum_{m=0}^jx^{2m}\right)^l\\
        &=\sum_{m_1=0}^j\sum_{m_2=0}^j\cdots\sum_{m_l=0}^j x^{2(m_1+m_2+\cdots+m_l)}\\
        &=\sum_{m=0}^{lj}c_mx^{2m},
   \end{align*}
    where $c_m$ denotes the number of ways of writing $m$ as the sum of $l$ integers $m_i$ with $0\leq m_i\leq j$. We can also see \eqref{E1} in \cite{Fahssi} through polynomial expansion. The formula \eqref{E2} for $c_m$ follows from \cite[Section 6]{Eger}. 
    
    To prove the Palindromic property, let $0\leq m\leq lj$. Then ,corresponding to every solution $ (m_1,m_2,\cdots,m_l)$ of $m$, we can always find a solution $ (j-m_1,j-m_2,\cdots,j-m_l)$ of $lj-m$ with $0\leq j-m_i\leq j$, and conversely. This makes a bijection between the set of solutions of $m$ and $lj-m$, implying $c_m=c_{lj-m}$. 
    
    Now, to prove the chain of inequalities, assume $lj$ is even and let $0\leq  m< \frac{lj}{2}$. For every solution $(m_1,m_2,\cdots,m_l)$ of $m$, note that not all $m_i$'s can be $j$, that is, at least one $m_i$ is less than $j$ because $m<\frac{lj}{2}$. Let $i_0$ be the least index such that $m_{i_0}<j$ and define the map
    \begin{equation*}
        \Phi(m_1,m_2,\cdots,m_l)=(m_1,\cdots,m_{i_0}+1,\cdots,m_l).
    \end{equation*}
    Then $\Phi(m_1,m_2,\cdots,m_l)$ is a solution of $m+1$ and the map $\Phi$ is injective, since from $\Phi(m_1,m_2,\cdots,m_l)$ one can uniquely recover $(m_1,m_2,\cdots,m_l)$ by subtracting 1 from the least index whose entry was increased. Hence $c_m\leq c_{m+1}$ for $0\leq m< \frac{lj}{2}$. The remaining inequalities follow from the Palindromic property. The other case, when $lj$ is odd, follows similarly.
\end{proof}

\begin{lemma}\label{L2.2}
     For each prime $p$, we have
     \begin{equation*}
         \lambda_{{\rm{sym}}^j}^l(p)=
         \begin{cases}
            \displaystyle \sum_{m=0}^{\frac{lj}{2}}d_m \lambda_{{\rm{sym}}^{lj-2m}}(p)&\text{ if $lj$ is even}\\
             \displaystyle \sum_{m=0}^{\frac{lj-1}{2}}e_m \lambda_{{\rm{sym}}^{lj-2m}}(p)&\text{ if $lj$ is odd}
         \end{cases}
     \end{equation*}
     where 
     \begin{align*}
        &d_m:=c_m(l,j)-c_{m-1}(l,j)=\sum_{r=0}^{\left\lfloor \frac{m}{\,j+1\,} \right\rfloor}(-1)^{r}\binom{l}{r}\binom{m-r(j+1)+l-2}{\,l-2\,}\\
        &e_m:=c_m(l,j)-c_{m-1}(l,j)=\sum_{r=0}^{\left\lfloor \frac{m}{\,j+1\,} \right\rfloor}(-1)^{r}\binom{l}{r}\binom{m-r(,j+1)+l-2}{\,l-2\,}
     \end{align*}
    with $c_{-1}=0$. Note that even though $d_m$ and $ e_m$ are given by the same formula, they are different as they are evaluated for different parities.
\end{lemma}
\begin{proof}
    We know that for each prime $p$ and integers $j\geq1$, we have 
    \begin{equation*}
        \lambda_{{\rm{sym}}^jf}(p)=\lambda_f(p^j)=\sum_{m=0}^j\alpha_p^{j-m}\beta_p^m.
    \end{equation*}
    Assume $lj$ is even. Following the above lemma, we obtain
    \begin{align*}
        \lambda_{{\rm{sym}}^jf}^l(p)=&\left(\sum_{m=0}^j\alpha_p^{j-m}\beta_p^m\right)^l=\alpha_p^{lj}\left(\sum_{m=0}^j\beta_p^{2m}\right)^l\\
        =&\alpha_p^{lj}\sum_{m=0}^{lj}c_m\beta_p^{2m}=\sum_{m=0}^{lj}c_m\alpha_p^{lj-m}\beta_p^{m}\\
        =&\sum_{m=0}^{\frac{lj}{2}-1}c_m\alpha_p^{lj-m}\beta_p^{m}+c_{\frac{lj}{2}}\alpha_p^{\frac{lj}{2}
        }\beta_p^{\frac{lj}{2}}+\sum_{m=\frac{lj}{2}+1}^{lj}c_{lj-m}\alpha_p^{lj-m}\beta_p^{m}\\
        =&c_0\left(\alpha_p^{lj}+\alpha_p^{lj-2}\beta_p^2+\cdots+1+\cdots+\alpha_p^{2}\beta_p^{lj-2}+\beta_p^{lj}\right)\\
        &+(c_1-c_0)\left(\alpha_p^{lj-2}+\alpha_p^{lj-4}\beta_p^4+\cdots+1+\cdots+\alpha_p^{4}\beta_p^{lj-4}+\beta_p^{lj-2}\right)\\
        &+(c_2-c_1)\left(\alpha_p^{lj-4}+\cdots+1+\cdots+\beta_p^{lj-4}\right)\\
        &+\cdots+\left(c_{\frac{lj}{2}-1}-c_{\frac{lj}{2}-2}\right)\left(\alpha_p^2+1+\beta_p^2\right)+\left(c_{\frac{lj}{2}}-c_{\frac{lj}{2}-1}\right)\\
        =&d_0\lambda_{{\rm{sym}}^{lj}f}(p)+d_1\lambda_{{\rm{sym}}^{lj-2}f}(p)+\cdots+d_{\frac{lj}{2}-1}\lambda_{{\rm{sym}}^{2}f}(p)+d_{\frac{lj}{2}}\\
        =&\sum_{m=0}^{\frac{lj}{2}}d_m \lambda_{{\rm{sym}}^{lj-2m}}(p),
    \end{align*}
    where 
     \begin{equation*}
         d_m=c_m-c_{m-1}=\sum_{r=0}^{\left\lfloor \frac{m}{\,j+1\,} \right\rfloor}
(-1)^{r}\binom{l}{r}
\binom{m-r(j+1)+l-2}{\,l-2\,},
     \end{equation*}
     since $\binom{n}{l-1}-\binom{n-1}{l-1}=\binom{n-1}{l-2}$.  The other case, when $lj$ is odd, follows similarly by the following equality
     \begin{align*}
        \lambda_{{\rm{sym}}^jf}^l(p)=\sum_{m=0}^{\frac{lj-3}{2}}c_m\alpha_p^{lj-m}\beta_p^{m}+c_{\frac{lj-1}{2}}\alpha_p^{2}+c_{\frac{lj+1}{2}}\beta_p^{2}+\sum_{m=\frac{lj+3}{2}}^{lj}c_{lj-m}\alpha_p^{lj-m}\beta_p^{m}.
     \end{align*}
\end{proof}

\begin{lemma}\label{L2.3}
For $\Re(s)>1$, we have
    \begin{align*}
        L_{l,j}(s):&=\sum_{n=1}^\infty\frac{\lambda_{{\rm{sym}}^{j}}^l(n)}{n^s}\\
        &=
        \begin{cases}
           \displaystyle \zeta^{d_{\frac{lj}{2}}}(s)\prod_{m=0}^{\frac{lj}{2}-1}L^{d_m}(s, {\rm{sym}}^{lj-2m}f)U_{l,j}(s)& \text{if $lj$ is even}\\
           \displaystyle\prod_{m=0}^{\frac{lj-1}{2}}L^{e_m}(s, {\rm{sym}}^{lj-2m}f)V_{l,j}(s)& \text{if $lj$ is odd}
        \end{cases}
    \end{align*}
    where $U_{l,j(s)}$ and $V_{l,j(s)}$ are some harmless Dirichlet serieses, which converges absolutely in the half plane $\Re(s)\geq \frac{1}{2}+\varepsilon$ and $U_{l,j}(1)\neq 0$, $V_{l,j}(1)\neq 0$. Moreover, the degree of the $L$-function $L_{l,j}(s)$ is $ D:=(j+1)^l$.
\end{lemma}
\begin{proof}
    Since the Fourier coefficients $\lambda_{{\rm{sym}}^{j}}(n)$ are multiplicative and satisfies the bound $\lambda_{{\rm{sym}}^{j}}(n)\ll n^\varepsilon$, for $\Re(s)>1$ we have the Euler product representation of $ L_{l,j}(s)$ as
    \begin{equation*}
         L_{l,j}(s)=\prod_p\left(1+\frac{\lambda_{{\rm{sym}}^{j}}^l(p)}{p^s}+\frac{\lambda_{{\rm{sym}}^{j}}^l(p^2)}{p^{2s}}+\cdots\right). 
    \end{equation*}
    Note that the series 
    \begin{equation*}
        \frac{\lambda_{{\rm{sym}}^{j}}^l(p^2)}{p^{2s}}+ \frac{\lambda_{{\rm{sym}}^{j}}^l(p^3)}{p^{3s}}+\cdots
    \end{equation*}
    converges absolutely in the half plane $\Re(s)\geq \frac{1}{2}+\varepsilon$. Thus, the coefficient of $p^{-s}$ in the Euler product above decides the decomposition of the $L$-function $L_{l,j}(s)$. Thus, the lemma follows from the above lemma.

    Assume $lj$ is even. We know that the degree of the $L$-function $L(s, {\rm{sym}}^{j}f)$ is $j+1$, so \begin{align*}
        D=&\sum_{m=0}^{\frac{lj}{2}}d_m(lj-2m+1)=\sum_{m=0}^{\frac{lj}{2}}(c_m-c_{m-1})(lj-2m+1)\\
        =&c_0(lj+1)+(c_1-c_0)(lj-1)+(c_2-c_1)(lj-3)\\
        &+\cdots+\left(c_{\frac{lj}{2}-1}-c_{\frac{lj}{2}-2}\right)(2+1)+\left(c_{\frac{lj}{2}}-c_{\frac{lj}{2}-1}\right)\\
        =&2c_0+2c_1+\cdots+2c_{\frac{lj}{2}-1}+c_{\frac{lj}{2}}\\
        =&c_0+c_1+\cdots+c_{\frac{lj}{2}-1}+c_{\frac{lj}{2}}+c_{\frac{lj}{2}+1}+\cdots+c_{lj-1}+c_{lj}\\
        =&(j+1)^l,
    \end{align*}
    which follows from Lemma \ref{L2.1} by taking $x=1$. The other case, when $lj$ is odd, follows similarly.
\end{proof}
\begin{note}
    Note that $c_0=1, c_1=l$, $c_2=\frac{l(l+1)}{2}$ and $\displaystyle c_m=\binom{m+l-1}{l-1}$ for $m\leq j$ and $l\geq 2$. The following are the values of $c_m$'s and $d_m$'s for $0\leq m\leq \frac{lj}{2}$ for some $l$ and $j$ (pertaining to Theorem \ref{T1}). For $j=2$, and $2\leq l\leq 8$, we have
   \begin{align*}
         l=2,\longrightarrow &c_m:1,2,3;\\
         &d_m:1,1,1\\
         l=3,\longrightarrow &c_m:1,3,6,7;\\
         &d_m:1,2,3,1\\
         l=4,\longrightarrow &c_m:1,4,10,16,19;\\
         &d_m:1,3,6,6,3\\
         l=5,\longrightarrow &c_m:1,5,15,30,45,51;\\
         &d_m:1,4,10,15,15,6\\
         l=6,\longrightarrow &c_m:1,6,21,50,90,126,141;\\
         &d_m:1,5,15,29,40,36,15\\
         l=7,\longrightarrow &c_m:1,7,28,77,161,266,357,393;\\
         &d_m:1,6,21,49,84,105,91,36\\
         l=8,\longrightarrow &c_m:1,8,36,112,266,504,784,1016,1107;\\
         &d_m:1,7,28,76,154,238,280,232,91.
   \end{align*}
    For $l=2$, and $2\leq j\leq 8$, we have $c_m=m+1$ for $0\leq m\leq \frac{lj}{2}$ and  so, $d_m=1$ for $0\leq m\leq \frac{lj}{2}$.
   \label{N2.1}
\end{note}

\begin{lemma}\label{L2.4}
Suppose that $\mathfrak{L}(s)$ is a general $L$-function of degree $m$. Then for any $\varepsilon>0$, we have
		\begin{equation}
			\int_T^{2T}\mid\mathfrak{L}(\sigma+it)\mid^2dt\ll T^{\max\{m(1-\sigma),\ 1\}+\varepsilon}\label{E3}
		\end{equation}
		uniformly for $\frac{1}{2}\leq \sigma\leq 1$ and $T>1$; and 
		\begin{equation}
			\mathfrak{L}(\sigma+it)\ll (10+ \mid t\mid)^{\frac{m}{2}(1-\sigma)+\varepsilon}\label{E4}
		\end{equation}
		uniformly for $\frac{1}{2}\leq \sigma\leq 1+\varepsilon$ and $\mid t\mid>10$.
	\end{lemma}
	\begin{proof}
		The result \eqref{E3} follows from Perelli \cite{Perelli}, and \eqref{E4} follows from the maximum modulus principle.
	\end{proof}

\begin{lemma}\label{L2.5}
    Let $K=\frac{8\sqrt{15}}{63}$. Then for any $\varepsilon>0$, we have
    \begin{equation}
        \zeta(\sigma+it)\ll |t|^{K(1-\sigma)^\frac{3}{2}+\varepsilon}\label{E5}
    \end{equation}
    uniformly for $|t|\geq 10$ and $\frac{1}{2}\leq\sigma\leq 1$.\label{L2.8}
\end{lemma}
\begin{proof}
    The result is due to Heath-Brown. See \cite{Heath-Brown}.
\end{proof}

\begin{lemma}\label{L2.6}
		For $\frac{1}{2}\leq \sigma\leq 2$, $T$ sufficiently large, there exists a $T^*\in[T,T+T^\frac{1}{3}]$ such that 
		\begin{equation*}
			\log\zeta(\sigma+iT^*)\ll (\log\log T^*)^2\ll(\log\log T)^2
		\end{equation*}
		holds. Thus, we have 
		\begin{equation}
			\mid \zeta(\sigma+it)\mid \ll \exp((\log\log T^*)^2)\ll T^\varepsilon\label{E6}
		\end{equation}
		on the horizontal line with $t=T^*$ uniformly for $\frac{1}{2}\leq \sigma\leq 2$.\label{L2.9}
	\end{lemma}
	\begin{proof}
		See Lemma $1$ of \cite{KRAS}.
	\end{proof}

\begin{lemma}\label{L2.7}
    For any $\varepsilon>0$, we have
    \begin{align}
			&L(\sigma+it, {\rm{sym}}^2f)\ll (10+\mid t\mid)^{\max \{\frac{6}{5}(1-\sigma), 0\}+\varepsilon};\label{E7} \\
            &L(\sigma+it, f)\ll (10+\mid t\mid)^{\max \{\frac{2}{3}(1-\sigma), 0\}+\varepsilon}\label{E8}
		\end{align}
        holds uniformly for $\frac{1}{2}\leq \sigma\leq 1+\varepsilon$ and $\mid t\mid\geq 10$.
\end{lemma}
\begin{proof}
    See \cite{LinNunesQi} and \cite{Ivic2012}.
\end{proof}
Now, we are ready to prove our main theorem.
\section{Proof of Theorem \ref{T1}}
Firstly, we consider the case where $lj$ is even and $lj\geq 6$. Applying Perron's formula to $L_{l,j}(s)$, we get 
\begin{equation*}
     \sum_{n\leq x} \lambda_{{\rm{sym}}^j}^l(n)=\frac{1}{2\pi i}\int_{1+\varepsilon-iT}^{1+\varepsilon+iT}L_{l,j}(s)\frac{x^s}{s}ds+O\left(\frac{x^{1+\varepsilon}}{T}\right),
\end{equation*}
where $10\leq T\leq x$ is a parameter that will be chosen later. We now move the line of integration to $\Re(s)=1-\frac{1}{j^3}$. Then in the rectangle formed by the vertices $1+\varepsilon+iT, 1-\frac{1}{j^3}+iT, 1-\frac{1}{j^3}-iT, 1+\varepsilon-iT$, the $L$-function $L_{l,j}(s)$ has a pole at $s=1$ of order $d_{\frac{lj}{2}}$ coming from the factor $\zeta^{d_{\frac{lj}{2}}}(s)$. The Cauchy residue theorem implies
\begin{align*}
     \sum_{n\leq x} \lambda_{{\rm{sym}}^j}^l(n)=&\frac{1}{2\pi i}\left\{\int_{1-\frac{1}{j^3}-iT}^{1-\frac{1}{j^3}+iT}+\int_{1-\frac{1}{j^3}+iT}^{1+\varepsilon+iT}+\int_{1+\varepsilon-iT}^{1-\frac{1}{j^3}-iT}\right\}L_{l,j}(s)\frac{x^s}{s}ds\\
     &+xP_{d_{\frac{lj}{2}}-1}(\log x)+O\left(\frac{x^{1+\varepsilon}}{T}\right)\\
     =&xP_{d_{\frac{lj}{2}}-1}(\log x)+I_{l,j}^1+I_{l,j}^2+I_{l,j}^3+O\left(\frac{x^{1+\varepsilon}}{T}\right),
\end{align*}
where $P_{d_{\frac{lj}{2}}-1}(y)$ is a polynomial of degree $d_{\frac{lj}{2}}-1$ in $y$.

Here we make the special choice $T=T^*$ of Lemma \ref{L2.9}, which satisfies \eqref{E6}, so that the horizontal portions $I_{l,j}^2$ and $I_{l,j}^3$ are controlled by the vertical line contribution $I_{l,j}^1$. The contribution of $I_{l,j}^1$ is given by 
\begin{align*}
    I_{l,j}^1\ll& x^{1-\frac{1}{j^3}+\varepsilon}+\int_{10}^T\left|L_{l,j}(1-\frac{1}{j^3}+it)\right|x^{1-\frac{1}{j^3}+\varepsilon} t^{-1}dt\\
    \ll & x^{1-\frac{1}{j^3}+\varepsilon}+x^{1-\frac{1}{j^3}+\varepsilon}\times\\
    &\int_{10}^T |\zeta^{d_{\frac{lj}{2}}}(1-\frac{1}{j^3}+it)\prod_{m=0}^{\frac{lj}{2}-1}L^{d_m}(1-\frac{1}{j^3}+it, {\rm{sym}}^{lj-2m}f)|t^{-1}dt\\
    \ll&  x^{1-\frac{1}{j^3}+\varepsilon}+x^{1-\frac{1}{j^3}+\varepsilon}\sup_{10\leq T_1\leq T} \left(\int_{T_1}^{2T_1}\left|L^{d_{\frac{lj}{2}-2}}(1-\frac{1}{j^3}+it, {\rm{sym}}^{4}f)\right|^2dt\right)^\frac{1}{2}\times\\
    &\left(\int_{T_1}^{2T_1}\prod_{m=0}^{\frac{lj}{2}-3}\left|L^{d_m}(1-\frac{1}{j^3}+it, {\rm{sym}}^{lj-2m}f)\right|^2dt\right)^\frac{1}{2}\times\\
    & \max_{T_1\leq t\leq 2T_1} \left|\zeta^{d_{\frac{lj}{2}}}(1-\frac{1}{j^3}+it)L^{d_{\frac{lj}{2}-1}}(1-\frac{1}{j^3}+it, {\rm{sym}}^{2}f)\right|T_1^{-1}\\
    \ll& x^{1-\frac{1}{j^3}+\varepsilon} T^{A_{l,j}+\varepsilon},
\end{align*}
where 
\begin{align*}
    A_{l,j}&=\frac{5}{2j^3}d_{\frac{lj}{2}-2}+\frac{1}{2j^3}\sum_{m=0}^{\frac{lj}{2}-3}(lj-2m+1)d_m+d_{\frac{lj}{2}}\left(\frac{8\sqrt{15}}{63}\right)\left(\frac{1}{j^3}\right)^\frac{3}{2}\\
    &\qquad+\frac{6}{5j^3}d_{\frac{lj}{2}-1}-1\\
    &=\frac{1}{2j^3}\sum_{m=0}^{\frac{lj}{2}-2}(lj-2m+1)d_m+d_{\frac{lj}{2}}\left(\frac{8\sqrt{15}}{63}\right)\left(\frac{1}{j^3}\right)^\frac{3}{2}+\frac{6}{5j^3}d_{\frac{lj}{2}-1}-1\\
    &=\frac{1}{2j^3}(D-d_{\frac{lj}{2}}-3d_{\frac{lj}{2}-1})+d_{\frac{lj}{2}}\left(\frac{8\sqrt{15}}{63}\right)\left(\frac{1}{j^3}\right)^\frac{3}{2}+\frac{6}{5j^3}d_{\frac{lj}{2}-1}-1,
\end{align*}
which follows from the Lemmas \ref{L2.4}, \ref{L2.5}, and \ref{L2.7}, and here $ D=(j+1)^l$.

The contribution of the horizontal portions $I_{l,j}^2$ and $I_{l,j}^3$, is given by
\begin{align*}
    I_{l,j}^2+I_{l,j}^3\ll&\int_{1-\frac{1}{j^3}}^{1+\varepsilon}\left|L_{l,j}(\sigma+iT)\right| x^{\sigma}d\sigma\\
    \ll&x^{\sigma}\int_{1-\frac{1}{j^3}}^{1+\varepsilon}T^{d_{\frac{lj}{2}}\varepsilon+\frac{6}{5}d_{\frac{lj}{2}-1}(1-\sigma)+\sum_{m=0}^{\frac{lj}{2}-2}(lj-2m+1)\frac{d_m}{2}(1-\sigma)-1} d\sigma\\
    \ll& x^{1-\frac{1}{j^3}+\varepsilon}T^{B_{l,j}+\varepsilon}+\frac{x^{1+\varepsilon}}{T},
\end{align*}
which follows from the Lemmas \ref{L2.4}, \ref{L2.5}, \ref{L2.6} and \ref{L2.7}, where \[B_{l,j}=\frac{1}{2j^3}(D-d_{\frac{lj}{2}}-3d_{\frac{lj}{2}-1})+\frac{6}{5j^3}d_{\frac{lj}{2}-1}-1.\]
We see that $B_{l,j}<A_{l,j}$. Note that, to use the Cauchy-Schwarz inequality in the way above, we should have $lj\geq 6$, and the same split does not hold when $lj=4$, as the corresponding $L$-function contains at most 3 factors. Thus, we treat the case $lj=4$ separately. However, the contribution of the horizontal portions is valid for all $l$ and $j$. Therefore, we get 
\begin{align*}
     \sum_{n\leq x} \lambda_{{\rm{sym}}^j}^l(n)=&xP_{d_{\frac{lj}{2}}-1}(\log x)+O\left(x^{1-\frac{1}{j^3}+\varepsilon}T^{A_{l,j}+\varepsilon}\right)+O\left(\frac{x^{1+\varepsilon}}{T}\right).
\end{align*}

Now, we make our choice of $T$ as suitable as $x^{1-\frac{1}{j^3}+\varepsilon} T^{A_{l,j}}=\frac{x}{T}$. That is,
\begin{align*}
    x=&T^{\frac{1}{2}(D-d_{\frac{lj}{2}}-3d_{\frac{lj}{2}-1})+d_{\frac{lj}{2}}\left(\frac{8\sqrt{15}}{63}\right)\left(\frac{1}{j^\frac{3}{2}}\right)+\frac{6}{5}d_{\frac{lj}{2}-1}}\\
    =&T^{\frac{j^\frac{3}{2}\left(315D-315d_{\frac{{lj}}{2}}-189d_{\frac{{lj}}{2}-1}\right)+80\sqrt{15}d_\frac{lj}{2}}{630j^\frac{3}{2}}},
\end{align*}
imply $T=x^{\frac{630j^\frac{3}{2}}{j^\frac{3}{2}\left(315D-315d_{\frac{{lj}}{2}}-189d_{\frac{{lj}}{2}-1}\right)+80\sqrt{15}d_\frac{lj}{2}}}$.

Thus, we obtain
\begin{align*}
     \sum_{n\leq x} \lambda_{{\rm{sym}}^j}^l(n)=&xP_{d_{\frac{lj}{2}}-1}(\log x)+O\left(x^{\theta_{l,j}+\varepsilon}\right),
\end{align*}
where $P_k(y)$ is a polynomial of degree $k$ in $y$ and 
\begin{equation*}
    \theta_{l,j}=1-\frac{630j^\frac{3}{2}}{j^\frac{3}{2}\left(315D-315d_{\frac{{lj}}{2}}-189d_{\frac{{lj}}{2}-1}\right)+80\sqrt{15}d_\frac{lj}{2}}.
\end{equation*}

We now consider the case $lj=4$. In this case, the $L$-function is
\begin{equation*}
    L_{l,j}(s)=\zeta^{d_2}(s)L^{d_1}(s, {\rm{sym}}^{2}f)L^{d_0}(s, {\rm{sym}}^{4}f)U_{l,j}(s).
\end{equation*}
By applying the Perron's formula to $ L_{l,j}(s)$, we get
\begin{equation*}
     \sum_{n\leq x} \lambda_{{\rm{sym}}^j}^l(n)=\frac{1}{2\pi i}\int_{1+\varepsilon-iT}^{1+\varepsilon+iT} L_{l,j}(s)\frac{x^s}{s}ds+O\left(\frac{x^{1+\varepsilon}}{T}\right).
\end{equation*}
Here, we move the line of integration to $\Re(s)=1-\frac{1}{j^3}$. Then
\begin{align*}
     \sum_{n\leq x} \lambda_{{\rm{sym}}^j}^l(n)=&xP_{d_2-1}(\log x)+\frac{1}{2\pi i}\left\{\int_{1-\frac{1}{j^3}-iT}^{1-\frac{1}{j^3}+iT}+\int_{1-\frac{1}{j^3}+iT}^{1+\varepsilon+iT}+\int_{1+\varepsilon-iT}^{1-\frac{1}{j^3}-iT}\right\} L_{l,j}(s)\frac{x^s}{s}ds\\
     &+O\left(\frac{x^{1+\varepsilon}}{T}\right)\\
     =&xP_{d_2-1}(\log x)+I_{l,j}^1+I_{l,j}^2+I_{l,j}^3+O\left(\frac{x^{1+\varepsilon}}{T}\right),
\end{align*}
where $P_{k}(y)$ is some polynomial of degree $k$ in $y$.

The vertical line $I_{l,j}^1$ contribution is given by
\begin{align*}
     I_{l,j}^1\ll& x^{1-\frac{1}{j^3}+\varepsilon}+\int_{10}^T\left| L_{l,j}(1-\frac{1}{j^3}+it)\right|x^{1-\frac{1}{j^3}+\varepsilon} t^{-1}dt\\
    \ll & x^{1-\frac{1}{j^3}+\varepsilon}+x^{1-\frac{1}{j^3}+\varepsilon}\times\\
    &\int_{10}^T \left|\zeta^{d_2}(1-\frac{1}{j^3}+it)L^{d_1}(1-\frac{1}{j^3}+it, {\rm{sym}}^2f)L^{d_0}(1-\frac{1}{j^3}+it, {\rm{sym}}^4f)\right|t^{-1}dt\\
    \ll&  x^{1-\frac{1}{j^3}+\varepsilon}+x^{1-\frac{1}{j^3}+\varepsilon}\sup_{10\leq T_1\leq T} \left(\int_{T_1}^{2T_1}\left|L^{d_1}(1-\frac{1}{j^3}+it, {\rm{sym}}^{2}f)\right|^2dt\right)^\frac{1}{2}\times\\
    &\qquad\left(\int_{T_1}^{2T_1}\left|L^{d_0}(1-\frac{1}{j^3}+it, {\rm{sym}}^{4}f)\right|^2dt\right)^\frac{1}{2}\max_{T_1\leq t\leq 2T_1} \left|\zeta^{d_2}(1-\frac{1}{j^3}+it)\right|T_1^{-1}\\
    \ll&  x^{1-\frac{1}{j^3}+\varepsilon}T^{d_{2}\frac{8\sqrt{15}}{63}\left(\frac{1}{j^3}\right)^{\frac{3}{2}}+\frac{D-d_2}{2}\frac{1}{j^3}-1}
\end{align*}
which follows from the Lemmas \ref{L2.4} and \ref{L2.5}. 

The contribution of $I_{l,j}^2$ and $I_{l,j}^3$ follows in a way similar to that of the case $lj\geq 6$, and note that $B_{l,j}<d_2\frac{8\sqrt{15}}{63}\left(\frac{1}{j^3}\right)^{\frac{3}{2}}+\frac{D-d_2}{2}\frac{1}{j^3}-1$. Therefore, we get
\begin{align*}
     \sum_{n\leq x} \lambda_{{\rm{sym}}^j}^l(n)=xP_{d_2-1}(\log x)+O\left(x^{1-\frac{1}{j^3}+\varepsilon}T^{d_2\frac{8\sqrt{15}}{63}\left(\frac{1}{j^3}\right)^{\frac{3}{2}}+\frac{D-d_2}{2}\frac{1}{j^3}-1+\varepsilon}\right)+O\left(\frac{x^{1+\varepsilon}}{T}\right).
\end{align*}
Now, we make our choice of $T$ by making $x^{1-\frac{1}{j^3}}T^{d_2\frac{8\sqrt{15}}{63}\left(\frac{1}{j^3}\right)^{\frac{3}{2}}+\frac{D-d_2}{2}\frac{1}{j^3}-1}=\frac{x}{T}$, that is, $T=x^{\frac{126j^{3/2}}{63(D-d_2)j^{3/2}+16\sqrt{15}d_2}}$, implying
\begin{align*}
     \sum_{n\leq x} \lambda_{{\rm{sym}}^j}^l(n)=xP_{d_2-1}(\log x)+O\left(x^{1-\frac{126j^{3/2}}{63(D-d_2)j^{3/2}+16\sqrt{15}d_2}+\varepsilon}\right),
\end{align*}
where $P_{k}(y)$ is some polynomial of degree $k$ in $y$.

The other case when $lj$ is odd follows similarly by moving the line of integration to $\Re(s)=1-\frac{1}{j^3}$ and using the equations \eqref{E3}, \eqref{E4}, and \eqref{E8}. Note that in this case, the $L$-function $L_{l,j}(s)$ has no poles; thus, there do not exist any main terms.

This completes the proof of Theorem \ref{T1}.\qed

\noindent
{\bf Acknowledgements}: The author acknowledges the financial support and research facilities provided by the Indian Institute of Science Education and Research (IISER) Berhampur, Department of Mathematical Sciences, during the tenure of the postdoctoral fellowship.

\end{document}